\documentclass[11pt]{amsart}
\usepackage{amsmath,amssymb,latexsym,soul,cite,mathrsfs}

\usepackage{color,enumitem,graphicx}
\usepackage[colorlinks=true,urlcolor=blue,
citecolor=red,linkcolor=blue,linktocpage,pdfpagelabels,
bookmarksnumbered,bookmarksopen]{hyperref}
\usepackage[english]{babel}

\usepackage[left=2.9cm,right=2.9cm,top=2.8cm,bottom=2.8cm]{geometry}
\usepackage[hyperpageref]{backref}

\usepackage[colorinlistoftodos]{todonotes}
\makeatletter
\providecommand\@dotsep{5}
\def\listtodoname{List of Todos}
\def\listoftodos{\@starttoc{tdo}\listtodoname}
\makeatother

\numberwithin{equation}{section}
\def\dis{\displaystyle}
\newtheorem{definition}{Definition}
\newtheorem{taggedtheoremx}{Theorem}
\newenvironment{taggedtheorem}[1]
 {\renewcommand\thetaggedtheoremx{#1}\taggedtheoremx}
 {\endtaggedtheoremx}

\def\R {{\rm I}\hskip -0.85mm{\rm R}}
\def\N {{\rm I}\hskip -0.85mm{\rm N}}

\def\dis{\displaystyle}

\newtheorem{theorem}{Theorem}[section]
\newtheorem{proposition}[theorem]{Proposition}
\newtheorem{lemma}[theorem]{Lemma}
\newtheorem{corollary}[theorem]{Corollary}

\newtheorem{remark}{Remark}[section]

\title[positive solutions for Kirchhoff  equations
via bifurcation methods]
{Existence results of positive solutions for \\ Kirchhoff type equations
via bifurcation methods}

\author[W. Cintra]{Willian Cintra}
\author[J. R. Santos Jr.]{Jo\~ao R. Santos J\'unior}
\author[G. Siciliano]{Gaetano Siciliano}
\author[A. Su\'arez]{Antonio Su\'arez}

\address[W. Cintra]{\newline\indent Faculdade de Matem\'atica
\newline\indent 
Instituto de Ci\^{e}ncias Exatas e Naturais
\newline\indent 
Universidade Federal do Par\'a
\newline\indent
Avenida Augusto corr\^{e}a 01, 66075-110, Bel\'em, PA, Brazil}
\email{\href{mailto: willian\_matematica@hotmail.com }{willian\_matematica@hotmail.com}}

\address[J. R. Santos Jr.]{\newline\indent Faculdade de Matem\'atica
\newline\indent 
Instituto de Ci\^{e}ncias Exatas e Naturais
\newline\indent 
Universidade Federal do Par\'a
\newline\indent
Avenida Augusto corr\^{e}a 01, 66075-110, Bel\'em, PA, Brazil}
\email{\href{mailto: joaojunior@ufpa.br }{joaojunior@ufpa.br}}

\address[G. Siciliano]{\newline\indent Departamento de Matem\'atica
\newline\indent 
Instituto de Matem\'atica e Estat\'istica
\newline\indent 
 Universidade de S\~ao Paulo 
\newline\indent 
Rua do Mat\~ao 1010,  05508-090, S\~ao Paulo, SP, Brazil }
\email{\href{mailto:sicilian@ime.usp.br}{sicilian@ime.usp.br}}

\address[A. Su\'arez]{\newline\indent Departamento de Ecuaciones Diferenciales y An{\'a}lisis Num{\'e}rico
\newline\indent 
Facultad de Matem{\'a}ticas
\newline\indent 
Universidad de Sevilla
\newline\indent
C/. Tarfia s/n, 41012, Sevilla, Spain.}
\email{\href{mailto: suarez@us.es}{suarez@us.es}}

\thanks{
Willian Cintra was partially supported by CAPES, Brazil. Jo\~ao R. Santos J\'unior was partially supported by CNPq-Proc. 302698/2015-9 and CAPES-Proc. 88881.120045/2016-01, Brazil. Antonio Su\'arez has been partially supported by MTM2015-69875-P (MINECO/FEDER, UE) and  CNPq-Proc. 400426/2013-7.
Gaetano Siciliano  was partially supported by
Fapesp, Capes and CNPq, Brazil. }


\begin{document}

\maketitle
\begin{abstract}
In this paper we address the following Kirchhoff type problem 
\begin{equation*}
    \left\{ \begin{array}{ll}
         -\Delta(g(|\nabla u|_2^2) u + u^r) = a u + b u^p& \mbox{in}~\Omega,  \\
         u>0& \mbox{in}~\Omega,  \\
         u= 0& \mbox{on}~\partial\Omega,
    \end{array} \right.
\end{equation*}
in a bounded and smooth domain $\Omega$ in $\R^{N}$. By using change of variables and bifurcation methods, we 
show, under suitable conditions on the parameters $a,b,p,r$ and
the nonlinearity $g$, the existence of positive solutions.
\end{abstract}
\maketitle

\section{Introduction}

In the recent paper \cite{JG} two of the authors studied the following 
Kirchhoff type problem
\begin{equation}\label{eq:JG}
    \left\{ \begin{array}{ll}
         -\text{div}(m(u,|\nabla u|_2^2)\nabla  u ) = f(x,u)& \mbox{in}~\Omega,  \\
         u= 0& \mbox{on}~\partial\Omega,
    \end{array} \right.
\end{equation}
in the smooth and bounded domain $\Omega\subset \R^{N}, N\geq1$,
where $f$ is a sublinear nonlinearity and $m:\R\times[0,+\infty)\to \R$
is a function such that, setting
$$m_{t}: s\in  \R\mapsto m(s,t)\in \R,$$
the following hold:
\begin{enumerate}[label=(m\arabic*),ref=m\arabic*,start=0]
\item\label{m_{0}} $m: \R\times[0,+\infty)\to \mathbb (0,+\infty)$ is  continuous;  \medskip
\item\label{m_{1}} there is $\mathfrak m >0$ such that $m(s, t)\geq \mathfrak m $ for all $s\in\R$ and $t\in[0,\infty)$;  
\medskip
\item\label{m_{2}} for each $t\in [0,+\infty)$ the map 
$m_{t}:\R\to(0,+\infty)$ is strictly decreasing in $(-\infty,0)$  and strictly increasing in $(0,+\infty)$. \medskip 
\end{enumerate}
In particular the class of such admissible $m$ is very huge. 
In \cite{JG} the problem was addressed by finding the fixed point of the map
$$S: t\in[0,+\infty) \mapsto \int_{\Omega}|\nabla u_{t}|^{2}dx\in (0,+\infty)$$
where $u_{t}$ is the unique solution of the auxiliary problem
\begin{equation}\label{eq:Pr}
    \left\{ \begin{array}{ll}
         -\text{div}(m_{t}(u)\nabla  u ) = f(x,u)& \mbox{in}~\Omega,  \\
         u= 0& \mbox{on}~\partial\Omega.
    \end{array} \right.
\end{equation}
Indeed, the main difficulty related to  the method used was  to guarantee existence, uniqueness
and a priori bound with respect to $t$ for the solutions of \eqref{eq:Pr}.
In particular, the uniqueness was obtained in virtue of the sublinearity condition of the nonlinearity
$f$. However, as remarked in  \cite[Remark 5]{JG}, the uniqueness of the solution to \eqref{eq:Pr}
was not necessary in order to employ the method:
since the solution of \eqref{eq:JG} is obtained as a fixed point of $S$,
all that is needed is the existence of a continuous branch of solutions $t\mapsto u_{t}$
to   \eqref{eq:Pr}, rather than the uniqueness of the solution $u_{t}$.

We point out here that similar methods to that in \cite{JG} has been used recently in
\cite{JG2} to deal with a biharmonic Kirchhoff operator.

\medskip

Motivated by the above remark, in this paper we study a Kirchhoff type problem
where the uniqueness of the solution to the related auxiliary problem is not expected, due to the fact that
the nonlinearity is not sublinear. 
Indeed, in this paper bifurcation methods are used in order to circumvent the lack of uniqueness
and obtain a continuum of positive solutions. We point out that bifurcation theory has been used by other authors to study Kirchhoff problems (see, for instance, \cite{ArcoyaAmbrosetti, FMSS, shi} and references therein) and arguments combining bifurcation theory with Bolzano Theorem were used to analyze non-local elliptic problems in \cite{ArcoyaLP}.

More specifically, we study here
%
%
%
%
%
%
the existence of classical positive solutions for the  following problem
with a logistic type nonlinearity:
\begin{equation}\label{P1}
    \left\{ \begin{array}{ll}
         -\Delta(g(|\nabla u|_2^2) u + u^r) = a u + b u^p& \mbox{in}~\Omega,  \\
         u>0& \mbox{in}~\Omega,  \\
         u= 0& \mbox{on}~\partial\Omega,
    \end{array} \right.
\end{equation}
where $g:[0,\infty) \rightarrow [0,\infty)$ is a continuous function satisfying 
minimal assumptions and
$a,b,r,p$ are constants such that $a>0$, $p,r >1$.
Observe that
problem  \eqref{P1} is  of  Kirchhoff type since the equation can be written as
$$-\text{div} \left( m(u, |\nabla u|_{2}^{2}) \nabla u\right) = a u + b u^p$$
just defining $m(u,|\nabla u|_{2}^{2}) := g(|\nabla u|_{2}^{2}) +r u^{r-1}.$
We point out here that, if we fix the value $|\nabla u|_{2}^{2}=t$ and define the map
for every $t\geq0$ 
$$m_{t}: s\in \mathbb R\mapsto m(s,t)\in \mathbb R$$
the map $m$ satisfies the  above conditions \eqref{m_{0}} and  \eqref{m_{1}}; moreover
 for suitable values of $r$ also \eqref{m_{2}} holds, falling down into the class of
$m$ permitted in \cite{JG}.

\medskip

In this paper we give sufficient conditions, depending on the parameters $a,b,r,p$
and on $g$ in such a way that problem \eqref{P1} admits a positive solution.
We have to say, however, that many other cases remain open.

Here are our main results.

\begin{taggedtheorem}{A}\label{th:A}
Assume that  
 $g:[0,\infty) \rightarrow [0,\infty)$ is a continuous function with $\inf_{s>0}g(s) >0$.
If   one of the following conditions is satisfied:
\begin{itemize}
\item[(i)] $b\leq 0$,\smallskip
\item[(ii)] $b>0$, $r=p$ and $b<\lambda_1$,\smallskip
\item[(iii)] $b>0$ and $r>p$,
\end{itemize}
then problem \eqref{P1} admits at least one  positive solution for each  $a>g(0)\lambda_1$.

Hereafter $\lambda_{1}$ denotes the first eigenvalue of the Laplacian in $\Omega$ under homogeneous Dirichlet boundary conditions.
\end{taggedtheorem}
\begin{taggedtheorem}{B}\label{th:B}
Assume that $g:[0,\infty) \rightarrow [0,\infty)$ is a bounded continuous function such that $g(0)>0$.
\begin{itemize}
    \item[(i)] If $r=p$ and $b>\lambda_1$ then problem \eqref{P1} admits at least one  positive solution for each $a< g(0)\lambda_1$. \smallskip
    \item[(ii)] If $b>0$, $1<p/r <(N+2)/(N-2)$ and $g(0)\lambda_1>\phi(s_0)$, then problem \eqref{P1}
     admits at least one  positive solution for each $a \in (\phi(s_0),g(0)\lambda_1)$. 
    
Here 
$$s_0 :=\left(\frac{b(p-1)}{\lambda_1(r-1)}\right)^{{1}/{(r-p)}}$$
is the maximum point of the function $\phi(s):=
\lambda_1s^{r-1}-bs^{p-1},  s\geq 0.$
    \end{itemize}
\end{taggedtheorem}

\begin{taggedtheorem}{C}\label{th:C}
Assume that $g:[0,\infty) \rightarrow [0,\infty)$ is a continuous and positive function. If $b=\lambda_1$ and $r=p<2$, then problem \eqref{P1} admits a  positive solution if, and only if, $a\in \lambda_1 R[g] :=\{\lambda_1g(s); s \geq 0\} \subset (0,\infty)$.
\end{taggedtheorem}

\medskip

The paper is organized as follows. In Section \ref{sec:Prelim} 
some preliminary results are given and a suitable parameter-dependent problem 
(see \eqref{Pa1}) is introduced in order to study the original problem \eqref{P1}.

Section \ref{sec:Change} is devoted to study  problem \eqref{Pa1}.
However to do that, by means of a  change of variables,  the problem is transformed into an equivalent
one which is studied with  bifurcation theory.

In Section \ref{sec:final} the proof of the main results is given.

\medskip

As a matter of notations,
we set for $k\in  \N, C^{k}_{0}(\overline\Omega):=C^{k}(\Omega)\cap C_{0}(\overline{\Omega})$ the set of  functions 
defined on $\overline \Omega$ which are of class $C^{k}$ in $\Omega$ and continuous up to the boundary satisfying
the homogeneous Dirichlet boundary condition. Hereafter $|\cdot|_{p}$ stands for the $L^{p}(\Omega)-$norm 
and the uniform norm will be denoted by $\|\cdot\|$.

\section{Preliminaries}\label{sec:Prelim}

Our approach to deal with  problem \eqref{P1}
 is to consider the following auxiliary problem
depending on the positive parameter $\lambda$:
\begin{equation}\label{Pa1}\tag{$P_{\lambda}$}
\left\{ \begin{array}{ll}
-\Delta(u/\lambda + u^r) = a u + b u^p& \mbox{in}~\Omega,  \\
u>0 & \mbox{in}~\Omega,\\
u= 0& \mbox{on}~\partial\Omega.
\end{array} \right.
\end{equation}
By a solution we mean a pair $(\lambda, u_{\lambda})\in (0,+\infty)\times C^{2}_{0}(\overline\Omega)$ satisfying \eqref{Pa1}
in the classical sense.

The idea is to find first 
an unbounded continuum, say $\Sigma_0$, of positive solution of (\ref{Pa1}) via bifurcation theory. Then, using the classical Bolzano Theorem, we will search for zeros of the map, 
	$$h: (\lambda,u)\in \Sigma_{0}\longmapsto  \frac{1}{\lambda} - g(|\nabla u|_2^2) \in \R$$
	which of course provides us a  solution of (\ref{P1}).

\medskip

In the following, given functions $A \in {C}^1(\overline{\Omega})$ and $B\in {C}(\overline{\Omega})$, satisfying
 $A(x) \geq A_0>0$ for $x\in \Omega$
and a suitable constant $A_{0}$, we will denote by
$$\sigma_1 [ -\text{div} (A(x) \nabla) + B(x)] $$
the principal  eigenvalue of the problem
\begin{equation*}
\left\{\begin{array}{ll}
    -\text{div} \left(A(x) \nabla u\right) + B(x)u =\lambda u &\mbox{in}~\Omega,  \\
     u=0&\mbox{on}~\partial\Omega. 
\end{array}  \right.
\end{equation*}
It is well-known (see for instance \cite{SMPbook,Djairo}) that this eigenvalue is increasing with respect to $A$ 
and $B$. When $A \equiv A_0$ is constant, we have $-\text{div} (A_0 \nabla) = -A_0 \Delta$, thus in this case we will 
write
$$\sigma_1 [ -A_0 \Delta + B(x)].$$
By simplicity, we set
$$\lambda_1:= \sigma_1 [ -\Delta].$$
Moreover, we will denote by $\varphi_1$ the positive eigenfunction associated to $\lambda_1$ with $\|\varphi_1\|=1.$

\medskip

The first result gives  necessary conditions on the parameters $a,b,r,p$ 
in order to get solutions to \eqref{Pa1}.
Here  the following function 
\begin{equation}\label{eq:phi}
\phi(s):=\lambda_1s^{r-1}-bs^{p-1}, \quad s\geq 0,
\end{equation}
plays an important role.
It easy to see that if $r>p$ (resp. $r<p$) $\phi$ is bounded below (resp. above) and the minimum (resp. maximum) 
 is attained at
\begin{equation}\label{eq:s0}
s_0 =\left(\frac{b(p-1)}{\lambda_1(r-1)}\right)^{\frac{1}{r-p}},
\end{equation}
and
$$
\phi(s_0)=\left(\frac{p-1}{\lambda_1}\right)^{\frac{p-1}{r-p}}\left(\frac{b}{r-1}\right)^{\frac{r-1}{r-p}}(p-r).
$$
With these considerations, we can show the next result.

\begin{lemma}\label{ne}
We have the following.
\begin{enumerate}
    \item[(a)] Suppose $b\leq0$.  Then (\ref{Pa1}) does not possess positive solution for $\lambda \leq \lambda_1/a$.\smallskip
    \item[(b)]Suppose $b>0$ and $r>p$. Then  (\ref{Pa1}) does not possess positive solution for $\lambda < \lambda_1/(a-\phi(s_0))$. Moreover, if $\lambda> \lambda_1/a$, then (\ref{Pa1}) does not possess positive solution $(\lambda,u_\lambda)$ with $\|u_\lambda\| < (b/\lambda_1)^{1/(r-p)}$. \smallskip
    \item[(c)]Suppose $b>0$ and $r=p$. If $b<\lambda_1$ (resp. $b>\lambda_1$), then (\ref{Pa1}) does not posses positive solution for $\lambda<\lambda_1/a$ (resp. $\lambda>\lambda_1/a$). Moreover, if $b = \lambda_1$ then (\ref{Pa1}) does not possess positive solution for $\lambda \neq \lambda_1/a$. \smallskip
    \item[(d)] Suppose $b>0$, $r<p$ and $a>\phi(s_0)$. Then (\ref{Pa1}) does not possess positive solution for $\lambda>\lambda_1/(a-\phi(s_0))$. Moreover, if $\lambda< \lambda_1/a$, then (\ref{Pa1}) does not possess positive solution $(\lambda,u_\lambda)$ with $\|u_\lambda\| < (b/\lambda_1)^{1/(r-p)}$.
\end{enumerate}
\end{lemma}
\begin{proof}


To prove (a) and (b) we argue as follows. Let $\varphi_1$ be the positive eigenfunction associated to $\lambda_1$ with $\|\varphi_1\|=1$. Multiplying (\ref{Pa1}) by $\varphi_1$, integrating in $\Omega$ and applying the formula of integration by parts we get
\begin{eqnarray*}
\int_\Omega \nabla\left(\dis\frac{u_\lambda}{\lambda}+u_{\lambda}^r\right)  \nabla \varphi_1 &=& \int_\Omega (au_{\lambda} + bu_{\lambda}^p)\varphi_1\\
\int_\Omega \lambda_1\left(\dis\frac{u_\lambda}{\lambda}+u_{\lambda}^r\right)  \varphi_1 &=& \int_\Omega (au_{\lambda} + bu_{\lambda}^p)\varphi_1.
\end{eqnarray*}
Hence, $(\lambda, u_\lambda)$ satisfies
\begin{equation}\label{000}
\int_\Omega u_\lambda\varphi_1\left[\lambda_1 u_\lambda^{r-1}-bu_\lambda^{p-1} - \left(a- \frac{\lambda_1}{\lambda} \right) \right]=0.
\end{equation}
If $b\leq 0$ the above equality implies that $a>\lambda_1/\lambda$. This implies (a).
 
On the other hand, assuming $b>0$ since  $r>p$, the function
$$\phi(s)= \lambda_1 s^{r-1} - b s^{p-1}, \quad s \geq 0.$$
is bounded below and by a direct calculation, $\min_{0\leq s < \infty}\phi(s)$ is negative and it is attained at
$$s_0 =\left(\frac{b(p-1)}{\lambda_1(r-1)}\right)^{\frac{1}{r-p}}.$$
Therefore, for all $\lambda>0$ such that
$$
\lambda < \frac{\lambda_1}{a-\phi(s_0)},
$$
we have
$$\phi(s_0)> a - \frac{\lambda_1}{\lambda} \quad \forall \lambda \in (0,\lambda_1/(a-\phi(s_0))].$$
Consequently, 
$$\lambda_1 s^{r-1} - b s^{p-1} - \left(a - \frac{\lambda_1}{\lambda} \right)> 0 \quad \forall s \geq 0, ~\lambda \in (0,\lambda_1/(a-\phi(s_0))]$$
and (\ref{000}) cannot be satisfied for any positive function $u_\lambda$, showing that $\lambda\geq \lambda_1/(a-\phi(s_0))$ is a necessary condition for the existence of positive solution of (\ref{Pa1}) with $b>0$ and $r>p$.\\
Now, suppose that $(\lambda,u_\lambda)$ is a positive solution of (\ref{Pa1}) with $\lambda> \lambda_1/a$. Then, it verifies (\ref{000}) and, hence,
\begin{equation}\label{0001}
\int_\Omega u_\lambda\varphi_1(\lambda_1 u_\lambda^{r-1}-bu_\lambda^{p-1}) = \left(a- \frac{\lambda_1}{\lambda} \right)\int_\Omega u_\lambda \varphi_1>0.
\end{equation}
If $\|u_\lambda\|<(b/\lambda_1)^{1/(r-p)}$, then
$$ \phi(u_\lambda(x)) = \lambda_1 u_\lambda(x)^{r-1}-bu_\lambda(x)^{p-1} <0 \quad \forall x \in \Omega.$$
and, therefore,
$$
\int_\Omega u_\lambda\varphi_1(\lambda_1 u_\lambda^{r-1}-bu_\lambda^{p-1}) < 0,
$$
which is a contradiction with (\ref{0001}). 

\medskip

In the case (c), since $r=p$, (\ref{0001}) is equivalent to
$$
(\lambda_1-b)\int_\Omega u_\lambda^r\varphi_1 = \left(a- \frac{\lambda_1}{\lambda} \right)\int_\Omega u_\lambda \varphi_1
$$
whence we easily infer the result.

\medskip

The proof of paragraph (d) is similar to that of (b).
\end{proof}

In the next Section we will analyze the auxiliary problem \eqref{Pa1}.

\section{Study of the auxiliary problem (\ref{Pa1})} \label{sec:Change}
To study (\ref{Pa1}) we introduce the following change of variables depending on the parameter $\lambda$:
\begin{equation}\label{eq:cambio}
w:=\frac{u}{\lambda} + u^r.
\end{equation}

Then, if we define the function $I_{\lambda}:s\in[0,+\infty)\mapsto \dis\frac{s}{\lambda} + s^r\in[0,+\infty)$, which 
is a smooth diffeomorphism, and denote its 
inverse with $q_{\lambda}$, we have
$$w = I_{\lambda}(u) =  \frac{u}{\lambda} + u^r \Longleftrightarrow q_{\lambda}(w)=u.$$

The first result of this section collects some important properties related to the map $q_{\lambda}$.
\begin{lemma}\label{lem:q}
The map $\lambda \in (0,\infty) \mapsto q_{\lambda}(s)$ is continuous and increasing, for all $s >0$. 

Now let $\lambda>0$ be fixed. Then the map $s\in(0,+\infty)\mapsto\dis\frac{q_{\lambda}(s)}{s}\in(0,+\infty)$
	has the following properties:
	\begin{enumerate}[label=(q\arabic*),ref=q\arabic*,start=1]
		\item\label{q1} it is decreasing and $q_{\lambda}(s)/s\leq \lambda$, \smallskip
		\item\label{q2}
		$\displaystyle \lim_{s \rightarrow 0^{+}} \dis\frac{q_{\lambda}(s)}{s} = \lambda,$ 
		\item\label{q3} $\displaystyle \lim_{s\to+\infty}\dis\frac{q_{\lambda}(s)}{s} =0. $
	\end{enumerate}
	On the other hand, the map $s\in(0,+\infty)\mapsto\dis\frac{q_{\lambda}(s)^{p}}{s}\in(0,+\infty)$
	satisfies:
	\begin{enumerate}[label=(q\arabic*),ref=q\arabic*,start=4]
		\item\label{q4} 
		$\displaystyle \lim_{s \rightarrow 0^{+}}\dis \frac{q_{\lambda}(s)^p}{s} = 0$,\smallskip
		\item\label{q5} it is  increasing if $p \geq r$, \smallskip
		\item \label{q6} $\displaystyle \lim_{s\to+\infty} \dis\frac{q_{\lambda}(s)^{p}}{s}=
		\begin{cases}
		+\infty &\mbox{ if } r<p, \\
		0 &\mbox{ if }r>p, \\
		1 &\mbox{ if } r=p.
		\end{cases}
		$
	\end{enumerate}
\end{lemma}
\begin{proof}
The continuity of the map $\lambda \in (0,\infty) \mapsto q_\lambda(s)$,  $s > 0$ follows by
the  definition of $q_\lambda$. To prove that it is increasing, that is,
\begin{equation}\label{eq:qlcrescente}
\lambda'< \lambda'' \Longrightarrow q_{\lambda'}(s)< q_{\lambda''}(s),
\end{equation}
we proceed as follows. For every $s>0$, we have:
\begin{eqnarray*}\label{eq:}
I_{\lambda''}(q_{\lambda'}(s)) &=& \frac{q_{\lambda'}(s)}{\lambda''} + q_{\lambda'}(s)^{r} \\
&<&  \frac{q_{\lambda'}(s)}{\lambda'} + q_{\lambda'}(s)^{r} \\
&=& I_{\lambda'}(q_{\lambda'}(s))\\ &=& s
\end{eqnarray*}
from which \eqref{eq:qlcrescente} easily follows.

Now, let $\lambda>0$ be fixed. 	Since $q_{\lambda} $ is the inverse function of $I_{\lambda}$,  it is increasing and  verifies  
\begin{equation*}\label{s}
s=I_{\lambda}(q_{\lambda}(s)) =  \frac{q_{\lambda}(s)}{\lambda} + q_{\lambda}(s)^r 
\end{equation*}
so that
$$ \frac{q_{\lambda}(s)}{s} = \frac{1}{1/\lambda+q_{\lambda}(s)^{r-1}} \leq \lambda,
$$
which proves \eqref{q1}.
	
Furthermore, since $q_{\lambda}(0)= 0$ and $r>1$ it follows that
$$
\lim_{s \rightarrow 0^{+} } \frac{q_{\lambda}(s)}{s} =\lim_{s \rightarrow 0^{+}} \frac{1}{1/\lambda+q_{\lambda}(s)^{r-1}}  = \lambda\quad \textrm{ and }\quad \lim_{s\to+\infty}\frac{q_{\lambda}(s)}{s}=\lim_{s\to+\infty} \frac{1}{1/\lambda+q_{\lambda}(s)^{r-1}} =0
	$$
	which gives \eqref{q2} and \eqref{q3}.
	
	Moreover, being $p>1$,
	$$\lim_{s \rightarrow 0^{+}} \frac{q_{\lambda}(s)^p}{s} =  \lim_{s \rightarrow 0^{+}} \frac{q_{\lambda}(s)}{s} q_{\lambda}
	(s)^{p-1} = 0,$$
	proving \eqref{q4}.
	
	Finally, the identity
	$$ \frac{q_{\lambda}(s)^p}{s} =  \frac{1}{q_{\lambda}(s)^{1-p}/\lambda+ q_{\lambda}(s)^{r-p}}$$
	gives \eqref{q5} and \eqref{q6}.
\end{proof}

	\begin{remark}\label{limuniforme}
We point out that, in the case $r>p>1$, for each $\overline{\lambda}>0$, the limit $\lim_{s \rightarrow \infty} q_\lambda(s)^p/s = 0$,  is uniform in $\lambda \in [\overline{\lambda},\infty)$. Indeed, for each $\delta>0$, to obtain 
	$$ 
	\frac{q_\lambda(s)^p}{s} \leq \delta \left(\Longleftrightarrow s \leq I_\lambda((\delta s)^{1/p}) = \frac{(\delta s)^{1/p}}{\lambda}+ (\delta s)^{r/p} \right)
	$$
	it is sufficient to choose $s>0$ (independent on $\lambda$) such that
	$$
	1 \leq \delta^r s^{(r/p)-1} \left(\leq \frac{\delta ^{1/p} s^{(1/p)-1}}{\lambda}+\delta^r s^{(r/p)-1}\right).
	$$
    It should be noted that, by a similar argument, the limit  $\lim_{s \rightarrow \infty} q_\lambda(s)/s = 0$ is also uniform in $\lambda \in [\overline{\lambda}, \infty).$     
    
    We will use these properties later to get a priori bounds of the positive solutions of (\ref{Pa2}), uniform with respect to $\lambda \in [\overline{\lambda},\infty)$.
	\end{remark}

\medskip

Thus, under the above change of variable \eqref{eq:cambio},  problem \eqref{Pa1} is equivalent to
\begin{equation}\label{Pa2}
    \left\{ \begin{array}{ll}
         -\Delta w = a q_{\lambda}(w) + b q_{\lambda}(w)^p& \mbox{in }\Omega,  \\
         w>0& \mbox{in }\Omega, \\
         w= 0& \mbox{on } \partial\Omega,
    \end{array} \right.
\end{equation}
in the sense that  $(\lambda,u_{\lambda})$
 is a  solution of 
\eqref{Pa1} if, and only if, $(\lambda, w_{\lambda}):=(\lambda,  u_{\lambda}/\lambda+u_{\lambda}^r)$ 
is a    solution of \eqref{Pa2}.

We observe explicitly that the map 
$$f: (\lambda, s)\in (0,+\infty)\times[0,+\infty)\longmapsto a q_{\lambda}(s) +bq_{\lambda}(s)^{p}\in \mathbb R$$
is locally Lipschitz.
Note that for every $\lambda>0$ there is always the trivial solution $w\equiv 0$ 
to \eqref{Pa2}.

\begin{remark}\label{rem:necesariaw}
	In virtue of Lemma  	\ref{ne} we have the following necessary conditions
	for the existence of solutions to \eqref{Pa2}. 
	\begin{enumerate}
    \item[(a)] Suppose $b\leq0$.  Then \eqref{Pa2} does not possess positive solution for $\lambda \leq \lambda_1/a$. \smallskip
    \item[(b)] Suppose $b>0$ and $r>p$. Then  \eqref{Pa2} does not possess positive solution for $\lambda < \lambda_1/(a-\phi(s_0))$. Moreover, If $\lambda> \lambda_1/a$, then \eqref{Pa2} does not possess positive solution $(\lambda,w_\lambda)$ with $$\|w_\lambda\| < \frac{a}{\lambda_1} \left(\frac{b}{\lambda_1}\right)^{1/(r-p)} + \left(\frac{b}{\lambda_1}\right)^{r/(r-p)}.$$ \smallskip
    \item[(c)]Suppose $b>0$ and $r=p$. If $b<\lambda_1$ (resp. $b>\lambda_1$), then \eqref{Pa2} does not posses positive solution for $\lambda<\lambda_1/a$ (resp. $\lambda>\lambda_1/a$). Moreover, if $b = \lambda_1$ then \eqref{Pa2} does not possess positive solution for $\lambda \neq \lambda_1/a$. \smallskip
    \item[(d)] Suppose $b>0$, $r<p$ and $a>\phi(s_0)$, then \eqref{Pa2} does not possess positive solution for $\lambda>\lambda_1/(a-\phi(s_0))$. Moreover, if $\lambda< \lambda_1/a$, then \eqref{Pa2} does not possess positive solution $(\lambda,w_\lambda)$ with 
    $$\|w_\lambda\| < \frac{a}{\lambda_1} \left(\frac{b}{\lambda_1}\right)^{1/(r-p)} + \left(\frac{b}{\lambda_1}\right)^{r/(r-p)}.$$
\end{enumerate}

\end{remark}

Problem \eqref{Pa2} will be studied with the help of bifurcation theory.

\begin{definition}
We say that $(\lambda_0,0)$ is a bifurcation point from the trivial solution of the equation in \eqref{Pa2} if there exists a sequence $(\lambda_n,w_n)$ of non-trivial solutions of \eqref{Pa2} such that
$$ (\lambda_n,w_n) \rightarrow (\lambda_0,0) \quad \mbox{as}~n \rightarrow \infty.$$
\end{definition}

Now, we will 
obtain an unbounded continuum of positive solutions of 
(\ref{Pa2}) emanating from the trivial solution  at $\lambda = \lambda_1/a$ and, hence, we prove a result of existence of positive solution of (\ref{Pa2}) and 
consequently of (\ref{Pa1}). 

To this end, consider the map $\mathfrak{F}:(0,\infty)\times {C}_0^1(\overline{\Omega}) \rightarrow {C}_0^1(\overline{\Omega})$ defined  by
$$
\mathfrak{F}(\lambda, w) = w - (-\Delta)^{-1}[aq_{\lambda}(w) +b q_{\lambda}(w)^p ],
$$
where $(-\Delta)^{-1}$ is the inverse of the Laplacian operator under homogeneous Dirichlet boundary conditions. 
The operator $\mathfrak{F}$ is of class $\mathcal{C}^1$ and  equation in (\ref{Pa2}) 
(including the boundary condition) can be written in the form
$$\mathfrak{F}(\lambda, w) =0.$$
Moreover, since we are interested only in positive solutions  with $\lambda>0$, we can consider any $\mathcal{C}^1$-extension of $\mathfrak{F}$ on $\R\times C_0^1(\overline{\Omega})\times C_0^1(\overline{\Omega})$. Thus, still denoting it by $\mathfrak{F}$, we have:

\begin{proposition}\label{bifur}
The value $\lambda=\lambda_1/a$ is the unique  bifurcation point  from the trivial solution to \eqref{Pa2} with $\lambda>0$.
Moreover, from $\lambda_{1}/a$  emanates  an unbounded continuum  in $(0,\infty)\times {C}_0^1(\overline{\Omega})$ of positive solutions $\widehat\Sigma_{0}$ of  (\ref{Pa2}). 
\end{proposition}
\begin{proof}
Thanks  \eqref{q2} and \eqref{q4} of  Lemma \ref{lem:q}, 
we have
$$\lim_{s \rightarrow 0}\frac{aq_{\lambda}(s) + bq_{\lambda}(s)^p}{s} = a\lambda. $$
Thus, the linearization of $\mathfrak{F}$ at $(\lambda,0),\lambda>0$, is given by
$$ \partial_w\mathfrak{F}(\lambda,0) = I_{\mathcal{C}^1_{0}(\overline \Omega)} - a\lambda(-\Delta)^{-1}.$$
Consequently, $\partial_w\mathfrak{F}(\lambda,0)$ is a Fredholm operator of index zero, analytic in $\lambda$. Moreover, its kernel verifies
$$
N[\partial_w\mathfrak{F}(\lambda_1/a,0)] = \mbox{span}[\varphi_1], 
$$ 
where $\varphi_1 > 0$ stands, as usual, for the principal eigenfunction associated to $\lambda_1$ and satisfying 
$\|\varphi_1\|=1$. Furthermore, by a standard argument,
\begin{equation}\label{rrr}
\partial_\lambda \partial_w \mathfrak{F}(\lambda_1/a,0)\varphi_1 \not\in R[\partial_w \mathfrak{F}(\lambda_1/a,0)], 
\end{equation}
that is, $\lambda = \lambda_1/a$ is a 1-transversal eigenvalue of the family $\partial_w \mathfrak{F}(\lambda,0)$, which is the transversality condition stated in \cite{RabC}. Therefore, we can apply the unilateral bifurcation theorem
(see  \cite[Theorem 6.4.3]{Bifbook})  to conclude the existence of  a continuum $\widehat{\Sigma}_0$ of positive solution of (\ref{Pa2}) satisfying one of the following non-excluding options: either
\begin{enumerate}
    \item[1.] $\widehat{\Sigma}_0$ is unbounded in $\R \times {C}_0^1(\overline{\Omega})$.
    \item[2.] There exists $\lambda_* \in \R$ such that $\lambda_* \neq \lambda_1/a$ and $(\lambda_*,0) \in \Sigma_0.$
    \item[3.] $\widehat{\Sigma}_0$ contains a point $(\lambda,w) \in \R \times (Y \setminus \{0\})$, where $Y$ is the complement of $N[\partial_w\mathfrak{F}(\lambda_1/a,0)]$ in ${C}_0^1(\overline{\Omega})$.
\end{enumerate}
Let us prove that 2. and 3. cannot be satisfied. 

If 2. occurs, then $\lambda_*$ is a bifurcation point of (\ref{Pa2}) from the trivial solution. Since the  bifurcation points of the trivial solution of (\ref{Pa2}) are, necessarily, simple eigenvalues of $-\Delta$ in $\Omega$ under homogeneous Dirichlet boundary conditions divided by $a$ and the only simple eigenvalue is $\lambda_1/a$, we must have $\lambda_* = \lambda_1/a$, which is a contradiction.

Suppose now that 3. occurs. Note that we can take $Y= R[\partial_w\mathfrak{F}(\lambda_1/a,0)]$. Indeed, since  $\partial_w\mathfrak{F}(\lambda_1/a,0)$ is a Fredholm operator of index zero, we have
$$\mbox{codim} R[\partial_w\mathfrak{F}(\lambda_1/a,0)] = \mbox{dim} N[\partial_w\mathfrak{F}(\lambda_1/a,0)] =1$$
and, hence, the complement of $R[\partial_w\mathfrak{F}(\lambda_1/a,0)]$ on ${C}_0^1(\overline{\Omega})$ is one dimensional. From the transversality condition (\ref{rrr}), we obtain
$$R[\partial_w\mathfrak{F}(\lambda_1/a,0)] \oplus N[\partial_w\mathfrak{F}(\lambda_1/a,0)] ={C}_0^1(\overline{\Omega}).$$
Now, observe that the function $w$ given in paragraph $3.$ is positive and 
$$w \in Y = R[\partial_w\mathfrak{F}(\lambda_1/a,0)] = R[I- \lambda_1(-\Delta)^{-1}]$$
which is impossible. Then $\widehat\Sigma_0$ is unbounded in $\R \times {C}_0^1(\overline{\Omega})$.
\end{proof}

With the aim of obtaining a priori estimates
for the positive solutions of  (\ref{Pa2}), we
first recall the following result, which applies to the following general problem
\begin{equation}\label{eq:Games}
\left\{ \begin{array}{ll}
-\Delta u =  f(\lambda, x,u)& \mbox{in }\Omega,  \\
u= 0& \mbox{on } \partial\Omega.
\end{array} \right.
\end{equation}
A solution for this problem is a pair $(\lambda,u)$. We denote with
\eqref{eq:Games}$_{\lambda}$ the above problem with $\lambda$ fixed.

\begin{theorem}(See \cite[Theorem 2.2.]{Gamez})\label{th:Gamez}
	Assume that $f$ is locally Lipschitz. Suppose that $I\subset \mathbb R$,
	is an interval  and let $\Sigma \subset  I \times  C^{2}_{0}(\overline\Omega)$ 
	be a connected set of solutions of 
	\eqref{eq:Games}. Consider a continuous map $\overline U: I\to C^{2}_{0}(\overline\Omega)$ 
	such that $\overline U(\lambda) $ is a super-solution of \eqref{eq:Games}$_{\lambda}$ 
	for every $\lambda\in I$,
	but not a solution.
	If $u_{0}\leq  (\not\equiv)\overline U(\lambda_{0})$ for some $(\lambda_{0}, u_{0})\in\Sigma$,
	then $u<\overline U(\lambda)$ in $\Omega$, for all $(\lambda,u)\in \Sigma$.
\end{theorem}

Then, coming back to our problem we have the following.

\begin{lemma}\label{bound}
	Let $(\lambda,w_\lambda)\in \widehat{\Sigma}_{0}$, where $\widehat\Sigma_{0}$
	is given in Proposition \ref{bifur} (and hence $w_{\lambda}$ a positive solution of \eqref{Pa2}).
	\begin{enumerate}
		\item[(a)] If $b< 0$, then 
		$$\|w_\lambda\| \leq c/\lambda + c^r\quad \forall \lambda>0,$$
		where $c := (-a/b)^{1/(p-1)}$.
		\item[(b)] If $b=0$, then there exists $c_0>0$  such that 
		$$\|w_\lambda \|\leq c_0 \quad \forall \lambda > \lambda_1/a.$$
		\item[(c)] If $b>0$ and $p/r<1$ then  there exists $c_0>0$  such that 
		$$\|w_\lambda \|\leq c_0 \quad \forall \lambda \in \mbox{Proj}_{\R} \widehat{\Sigma}_0.$$
		\item[(d)]If $b>0$, $b \neq \lambda_1$ and $p/r=1$ then for each compact subset $\Lambda \subset (0,\infty)$ there exists $c_0>0$  such that 
		$$\|w_\lambda \|\leq c_1 \quad \forall \lambda \in \Lambda.$$
		\item[(e)] If $b>0$ and $1<p/r<(N+2)/(N-2)$, then for each compact subset  $\Lambda \subset (0,\infty)$,
		there exists $c_2>0$ such that 
		$$\|w_\lambda\| \leq c_2 \quad  \forall\lambda\in \Lambda.$$
	\end{enumerate}
\end{lemma}
\begin{proof}
	To prove (a), let $x_M \in \Omega$ be the point where $w_\lambda$ attained its maximum on $\Omega$. Then,
	\begin{eqnarray*}
	0 &\leq& -\Delta w_\lambda(x_M) = aq_{\lambda}(w_\lambda(x_M)) + b q_{\lambda}(w_\lambda(x_M))^p\\
	0 &\leq& a+b q_{\lambda}(w_\lambda(x_M))^{p-1}.
	\end{eqnarray*}
	Since $b<0$, this inequality is equivalent to
	$$
	q_{\lambda}(w_\lambda(x_M)) \leq  (-a/b)^{1/(p-1)} =: c = q_{\lambda} (I_{\lambda} (c)),
	$$
	and, hence, 
	$$w_\lambda(x_M) \leq c/\lambda+c^r = I_{\lambda}(c),$$
	showing that $w_\lambda(x) \leq c/\lambda+c^r$ for all $x \in \Omega$.
	
	\medskip
	
	In the case $b=0$,  we will build a family $\overline{W}(\lambda)$ of supersolutions of \eqref{Pa2}
	for every $\lambda\in[\lambda_{1}/a,+\infty)$ and apply Theorem \ref{th:Gamez}.
	To this aim, let
	$e$ be the 
	unique (positive) solution of
	\begin{eqnarray*}\label{e}
		\left\{\begin{array}{rl}
			-\Delta e = 1 & \mbox{in}~ \widehat{\Omega}, \\
			e=0& \mbox{on}~ \partial \widehat{\Omega},
		\end{array}\right.
	\end{eqnarray*}
	for some regular domain $\Omega \subset\subset \widehat{\Omega}$; in particular
	$e_m := \min_{\overline{\Omega}}e>0$.
	Let $K>0$ be a constant big enough (independent on $\lambda$) such that
	\begin{equation}\label{eq:Kgrande}
	q_{\lambda_1/a}(Ke_m)^{r-1} > a\|e\| \quad
	and \ \  Ke(x)\geq w_{\lambda_{1}/a}(x) \quad \forall x\in \Omega.
	\end{equation}
	Then, we consider the map 
	$$\overline{W}:[\lambda_1/a, \infty) \rightarrow C^2_0(\overline{\Omega}) \quad
	\text{such that } \overline{W} (\lambda) = Ke.$$
	We will show that $\overline{W}(\lambda) = Ke$ is a supersolution of (\ref{Pa2})
	for every $\lambda\in[\lambda_{1}/a,+\infty)$, that is
	$$
	\forall \lambda\geq\lambda_{1}/a: \ \ K=-\Delta(Ke) \geq a q_\lambda(Ke)\quad \mbox{in } {\Omega},
	$$
	or equivalently
	$$
	\forall \lambda\geq\lambda_{1}/a: \ \ Ke \geq aq_\lambda(Ke)e \quad \mbox{in } {\Omega}.
	$$
	Using that 
	$I_\lambda(q_\lambda(s)) = q_\lambda(s)/\lambda + q_\lambda(s)^r =s$, for all $s \geq 0$,
	we are actually reduced to show that
	\begin{eqnarray*}
		\forall \lambda\geq\lambda_{1}/a: \ \ \frac{q_\lambda(Ke)}{\lambda} + q_\lambda(Ke)^r \geq a q_\lambda(Ke)e \quad \mbox{in } {\Omega},
	\end{eqnarray*}
	and then, since $e(x)\geq e_m>0$ in $\overline{\Omega}$, to prove that
	\begin{equation}\label{ssol}
	\forall \lambda\geq\lambda_{1}/a: \ \ \frac{1}{\lambda} + q_\lambda(Ke)^{r-1} \geq a e \quad \mbox{in}~\Omega.
	\end{equation}
	Now since $\lim_{s \rightarrow \infty} q_\lambda(s) = \infty$, $r>1$ and  $e_m>0$, 
	by  the monotonicity of $q_{\lambda}$ with respect to $\lambda$ (see \eqref{eq:qlcrescente})
	and the choice of $K$ (see \eqref{eq:Kgrande}) we obtain that, for every $\lambda\geq\lambda_{1}/a$:
	$$ \frac{1}{\lambda} +  q_\lambda(Ke)^{r-1} \geq \frac{1}{\lambda} +  q_\lambda(Ke_m)^{r-1} > q_{\lambda_1/a}(Ke_m)^{r-1} > a\|e\| \geq ae \quad \mbox{in } \overline{\Omega},$$
	showing that \eqref{ssol} is satisfied and hence, $\overline{W}(\lambda) = Ke$ is a supersolution,
	but not a solution, of \eqref{Pa2} for every $\lambda\in[\lambda_{1}/a,+\infty)$. 
	Due to the choice of $K$ satisfying \eqref{eq:Kgrande},
	all the hypotheses of  Theorem \ref{th:Gamez}  are satisfied and then 
	we have
	$$w_\lambda < \overline{W}(\lambda) = Ke\leq K \max_{\overline\Omega}e:=c_{0},$$
	completing the proof of (b).
	
	\medskip
	
To prove   (c) we argue as above, building a family  $\overline{W}(\lambda)$ of supersolutions of (\ref{Pa2}). We consider again the constant map
	$$\overline{W}:[\lambda_1/a, \infty) \rightarrow C^2_0(\overline{\Omega}) \quad
	\text{such that } \overline{W} (\lambda) = Ke.$$
	Then, $\overline{W}(\lambda) = Ke$ is a supersolution of (\ref{Pa2}) for every $\lambda \in \mbox{Proj}_{\R}\widehat{\Sigma}_0$ if
	$$
	\forall \lambda\in \mbox{Proj}_{\R}\widehat{\Sigma}_0: \ \ K \geq aq_\lambda(Ke) + b q_\lambda(Ke)^p\quad \mbox{in } {\Omega},
	$$
	or equivalently
	$$
	\forall \lambda\in \mbox{Proj}_{\R}\widehat{\Sigma}_0: \ \ 1 \geq a\frac{q_\lambda(Ke)}{Ke}e + b \frac{q_\lambda(Ke)^p}{Ke} e\quad \mbox{in } {\Omega}.
	$$
	Since $r>p$, the limits $\lim_{s \rightarrow \infty} q_\lambda(s)/s = \lim_{s \rightarrow \infty} q_\lambda(s)^p/s = 0$ are uniform in $\lambda \in \mbox{Proj}_{\R}\widehat{\Sigma}_0$ (see Remark \ref{limuniforme}), we can obtain $K>0$ large enough (independent on $\lambda$) such that $\overline{W}(\lambda) = Ke$ is a supersolution of (\ref{Pa2}), proving the result.
	
	\medskip
Now, let us prove (d). To this end, we consider two cases: $0<b <\lambda_1$ and $b>\lambda_1$. For the first case we argue as above, building a family  $\overline{W}(\lambda)$ of supersolutions of (\ref{Pa2}). Thus, let $0<b < \lambda_1$ be fixed. By the monotonicity properties of principal eigenvalue with respect to the domain, we can get a regular domain $\widehat{\Omega}$ such that 
	$$
	\Omega \subset \subset \widehat{\Omega} \quad \mbox{and} \quad b< \widehat{\lambda}_1<\lambda_1
	$$ 
	where $\widehat{\lambda}_1$ stands for the principal eigenvalue of $-\Delta$ in $\widehat{\Omega}$ under homogeneous Dirichlet boundary conditions. Define
	$$\overline{W}:[\lambda_1/a, \infty) \rightarrow C^2_0(\overline{\Omega}) \quad
	\text{such that } \overline{W} (\lambda) = K\widehat{\varphi}_1.$$
	We will show that $\overline{W}(\lambda) = K \widehat{\varphi}_1$ is a supersolution of (\ref{Pa2})
	for every $\lambda\in[\lambda_{1}/a,+\infty)$, that is
	$$
	\forall \lambda\geq\lambda_{1}/a: \ \ K\widehat{\lambda}_1 \widehat{\varphi}_1=-\Delta(K\widehat{\varphi}_1) \geq a q_\lambda(K\widehat{\varphi}_1) + b q(K\widehat{\varphi}_1)^p\quad \mbox{in } {\Omega},
	$$
	or equivalently
	\begin{equation}\label{ttt}
		\forall \lambda\geq\lambda_{1}/a: \ \ \widehat{\lambda}_1    \geq \frac{q_\lambda(K\widehat{\varphi}_1) + b q(K\widehat{\varphi}_1)^p}{K\widehat{\varphi}_1} \quad \mbox{in } {\Omega}.
	\end{equation}
    Since $p=r$
	$$
	\lim_{s \rightarrow\infty}  \frac{q_\lambda(s) + b q(s)^p}{s} = b
	$$
	(see (\ref{q6}) in Lemma \ref{lem:q}), we can chose $K>0$ large enough such that (\ref{ttt}) holds. By Theorem \ref{th:Gamez}, we obtain the result.\\
	Now, let $b>\lambda_1$. We will proceed by contradiction. If (d) fails, there exists $(\lambda_n,w_n) \in \widehat{\Sigma}_0$ such that $\lambda_n \in \Lambda \subset(0,\infty)$ and
	$$
	\|w_n\| \rightarrow \infty \quad \mbox{as }n \rightarrow \infty. 
	$$
	Moreover, up to a subsequence if necessary,
	$$
	\lambda_n \rightarrow \lambda^*>0.
	$$
	Define
	$$
	z_n:= \frac{w_n}{\|w_n\|}, \quad n \geq 1.
	$$
	Since $(\lambda_n,w_n)$ is a positive solution of (\ref{Pa2}), we obtain that $(\lambda_n, z_n)$ verifies
	\begin{equation}\label{zn}
	\left\{\begin{array}{ll}
	    -\Delta z_n = \displaystyle\frac{a q_{\lambda_n}(w_n) +b q_{\lambda_n}(w_n)^p}{\|w_n\|}  & \mbox{in }\Omega,  \\
	     z_n = 0&\mbox{on }\partial\Omega. 
	\end{array} \right.
	\end{equation}
	Multiplying this equation by $z_n$,  integrating in $\Omega$ and applying the formula of integration by parts gives
	$$
	\|z_n\|^2_{H_0^1} = \int_\Omega \frac{(a q_{\lambda_n}(w_n) +b q_{\lambda_n}(w_n)^p)z_n}{\|w_n\|}. 
	$$
	Since $r=p$, by Lemma \ref{lem:q}, $q_\lambda(s)^p \leq s$ and $q_\lambda(s) \leq \lambda s$, for all $s \geq 0$ and $\lambda>0$. Thus,
	$$
	\|z_n\|^2_{H_0^1} \leq \int_\Omega \frac{(a \lambda_nw_n +b w_n)z_n}{\|w_n\|} = \int_\Omega (a \lambda_n  + b )z_n^2 \leq (a \max_{n\geq1} \lambda_n +b) |\Omega|, 
	$$
	showing that $z_n$ is bounded in $H_0^1(\Omega)$. By elliptic regularity, $z_n$ is also bounded in $W^{2,m}(\Omega)$, $m>1$. Thus, it follows from  Morrey's compact embedding that, up to  subsequence if necessary,
	$$
	z_n \rightarrow z \quad \mbox{in }C(\overline{\Omega}).
	$$
	Let $\psi \in C_0^\infty(\Omega)$. Multiplying (\ref{zn}) by $\psi$,  integrating in $\Omega$ and applying the formula of integration by parts gives
	$$
		-\int_\Omega z_n \Delta \psi  = \int_\Omega \frac{(a q_{\lambda_n}(w_n) +b q_{\lambda_n}(w_n)^p)\psi}{\|w_n\|} =\int_\Omega \frac{(a q_{\lambda_n}(z_n\|w_n\|) +b q_{\lambda_n}(z_n\|w_n\|)^p)z_n\psi}{z_n\|w_n\|}
	$$
	In view of (\ref{q3}) and (\ref{q6}) of Lemma \ref{lem:q}, letting $n \rightarrow \infty$ in the above equality yields
	$$
		-\int_\Omega z \Delta \psi  = \int_\Omega b z\psi.
	$$
	Since $z \in H_0^1(\Omega)$, $z\geq 0$ in $\Omega$ and $\|z\| =1$, it follows that $b = \lambda_1$, which is a contradiction with the inicial assumption $b>\lambda_1$.
	\medskip
	
	Finally, to prove  (e) observe that
	\begin{eqnarray*}
		\frac{aq_{\lambda}(s) + b q_{\lambda}(s)^p}{s^{p/r}} &=& \frac{aq_{\lambda}(s) + b q_{\lambda}(s)^p}{(q_{\lambda}(s)/\lambda + q_{\lambda}(s)^r)^{p/r}}\\ &=&\frac{a}{\left(q_{\lambda}(s)^{1-\frac{r}{p}}/\lambda+q_{\lambda}( s)^{r- {r}/{p}}\right)^{p/r}} + \frac{b}{\left(q_{\lambda}( s)^{1-r}/\lambda+1\right)^{p/r}}
	\end{eqnarray*}
	which implies
	$$ \lim_{s \rightarrow +\infty} \frac{aq_{\lambda}(s) + b q_{\lambda}(s)^p}{s^{p/r}} = b>0.$$
	By a classical result of Gidas and Spruck, see \cite[Theorem 1.1]{gidas},
	we obtain the conclusion.
\end{proof}
\begin{remark}
Note that when $b<0$  we have a uniform bound with respect to $\lambda$ of the positive solutions
$u_{\lambda}$ of problem \eqref{Pa1}. Indeed from (a) of Lemma \ref{bound}, since
$$
\forall x\in \Omega:  \ I_{\lambda}(u_{\lambda}(x))=\frac{u_{\lambda}(x)}{\lambda} +
 u_{\lambda}^{r}(x) = w_{\lambda}(x) \leq \frac c\lambda +c^{r}=I_{\lambda}(c),
$$
we obtain $\|u_{\lambda}\|\leq c.$
\end{remark}

Now, we have a better information on the behaviour of the continuum of 
positive solutions emanating from the trivial solution at $\lambda_{1}/a$.

\begin{proposition}\label{comportamentoSigma}
Let $\widehat{\Sigma}_{0}$	be the continuum of positive solutions of  \eqref{Pa2} given in Proposition \ref{bifur}.
\begin{enumerate}
\item Suppose that one of the following conditions is satisfied: 
\begin{enumerate}
\item[(i)] $b\leq 0$;
\item[(ii)] $b>0$, $r=p$ and $b<\lambda_1$;
\item[(iii)] $b>0$ and $r>p$.
\end{enumerate}
Then $(\lambda_1/a,\infty) \subset \mbox{Proj}_{\R}\widehat{\Sigma}_{0}$.  
\medskip
\item Suppose that one of the following conditions is satisfied: 
\begin{enumerate}
    \item[(i)] $r=p$ and $b>\lambda_1$;
    \item[(ii)] $b>0$, $1<p/r <(N+2)/(N-2)$ and $a>\phi(s_0)$.
\end{enumerate}
Then $(0,\lambda_1/a) \subset \mbox{Proj}_{\R}\widehat{\Sigma}_{0}$ and \eqref{Pa2} does not possess positive solution for $\lambda$ large. \medskip
\item Suppose that $r=p$ and $b= \lambda_1$, then $\mbox{Proj}_{\R}\widehat{\Sigma}_{0} = \{\lambda_1/a\}$. \end{enumerate}
\end{proposition}
\begin{proof}
First, observe that in the cases (1) and (2), by  Lemma  \ref{bound}, the positive solutions of (\ref{Pa2}) are bounded in ${C}(\overline{\Omega})$ and, by elliptic regularity,  also 
in ${C}_0^1(\overline{\Omega})$. Thus, in view of Remark  \ref{rem:necesariaw}: 
\begin{enumerate}
\item If (i), (ii) or (iii) occurs, then \eqref{Pa2} does not possess positive solution for $\lambda$ positive and small. Hence,  $\widehat{\Sigma}_0$ has to be  unbounded with respect to large values of $\lambda$, and this gives the conclusion.
\item  If (i) or (ii)  occurs, then there does not exist positive solution for $\lambda$ large enough. Therefore, since there does not exist bifurcation point from infinity  of positive solutions of \eqref{Pa2} (by Lemma  \ref{bound}),  the inclusion $(0,\lambda_1/a) \subset \mbox{Proj}_{\R}\widehat{\Sigma}_{0}$ must be satisfied.
\end{enumerate}
Finally, the case $r=p$ and $b=\lambda_1$ is a direct consequence of Remark \ref{rem:necesariaw} (c).
\end{proof}

Figure \ref{bif} shows some admissible situations within the setting of Proposition \ref{comportamentoSigma}.

An immediate consequence of the above proposition is the following result of existence of positive solution of (\ref{Pa2}).

\begin{figure}
\centering
\includegraphics[scale=0.6]{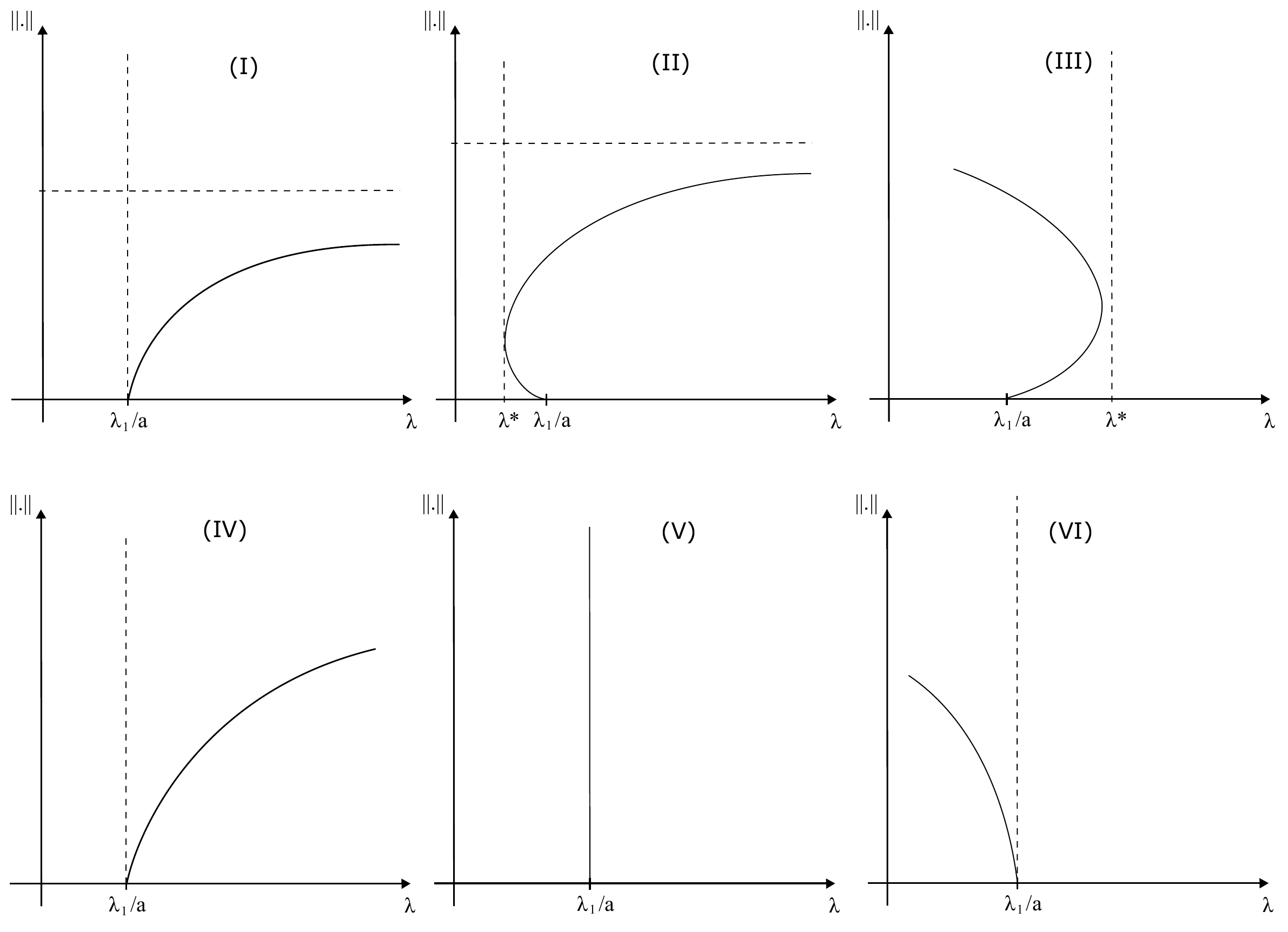}
\caption{\label{bif}Possible bifurcation diagrams of solution of (\ref{Pa2}) for the case (I) $b\leq0$; (II) $b>0$  and $r >p$; (III) $b>0$, $1<p/r< (N+2)(N-2)$ and $a>\phi(s_0)$; (IV) $0<b<\lambda_1$ and $r=p$; (V) $b=\lambda_1$ and $r=p$; (VI) $b>\lambda_1$ and $r=p$.}
\end{figure}

\begin{corollary}\label{Existencia}
We have the following.
\begin{enumerate}
    \item[(a)] Suppose that either $b \leq 0$ or $0<b<\lambda_1$ and $r=p$. Then (\ref{Pa2}) admits (at least) a positive solution if, and only if, $\lambda> \lambda_1/a$.  Moreover, assuming $\lambda > \lambda_1/a$, if \smallskip
    \begin{itemize}
    \item[(i)] $b=0$, or \smallskip
    \item[(ii)]$b<0$ and $p \geq r$ \smallskip
    \end{itemize}
    the solution is unique. \smallskip
    \item[(b)] Suppose $b>0$ and $r>p$. Then  (\ref{Pa2}) admits (at least) a positive solution if $\lambda> \lambda_1/a$. Moreover, there exists $\lambda^* \in (0,\lambda_1/a)$ such that (\ref{Pa2}) admits (at least) two positive solutions for each $\lambda \in (\lambda^*,\lambda_1/a)$.
    \item[(c)] Suppose that $r=p$ and $b>\lambda_1$. Then (\ref{Pa2}) admits (at least) a positive solution if  $0<\lambda<\lambda_{1}/a$.
    \item[(d)] Suppose that $b>0$, $r<p$, $p/r <(N+2)/(N-2)$ and $a>\phi(s_0)$. Then (\ref{Pa2}) admits (at least) a positive solution if  $0<\lambda<\lambda_{1}/a$. Moreover, there exists $\lambda^* >\lambda_1/a$ such that (\ref{Pa2}) admits (at least) two positive solutions for each $\lambda \in (\lambda_1/a, \lambda^*)$.
    \item[(e)] Suppose $r=p$ and $b=\lambda_1$. Then (\ref{Pa2}) admits a positive solution if, and only if, $\lambda = \lambda_1/a$. Moreover, all positive solutions of (\ref{Pa2}) are given by
    $$
    w=c\varphi_1 \quad c\in \R, c>0.
    $$
\end{enumerate}
\end{corollary}
\begin{proof} 
(a)
The existence is a consequence of Proposition \ref{bifur} and Remark \ref{rem:necesariaw} (a).
The  uniqueness follows by the same arguments as in 
 \cite[Section 2]{BO}, once that, in case (i) $s \mapsto aq(\lambda,s)/s$ is decreasing 
by Lemma \ref{lem:q} item \eqref{q1}, and in case (ii) 
  $s \mapsto (aq(\lambda,s)+b q(\lambda,s)^p)/s$ is 
 decreasing if $p \geq r$ by  Lemma \ref{lem:q} item \eqref{q5}. \medskip
 
(b) It follows again by  Proposition \ref{bifur} combined with Remark \ref{rem:necesariaw} (b) that ensures that the bifurcation at $\lambda=\lambda_1/a$ is subcritical. Indeed, for $\lambda>\lambda_1/a$ there does not exist positive solution of (\ref{Pa2}) with small norm. Hence, the bifurcation is subcritical. 

\medskip

The proofs of (c) and (d) are similar.

\medskip

In the case (e),  Remark \ref{rem:necesariaw} (c) ensures that $\lambda= \lambda_1/a$ is a necessary condition for the existence of positive solution of (\ref{Pa2}). Moreover, for all $c>0$,
\begin{eqnarray*}
-\Delta(c\varphi_1) = \lambda_1 c\varphi_1 &=& \lambda_1 I(q_{\lambda_1/a} (c\varphi_1))\\
&=& \lambda_1 \left( \frac{q_{\lambda_1/a} (c\varphi_1)}{\lambda_1/a} + q_{\lambda_1/a} (c\varphi_1)^p\right)\\
&=& aq_{\lambda_1/a} (c\varphi_1) + \lambda_1 q_{\lambda_1/a} (c\varphi_1)^p,
\end{eqnarray*}
showing that $w=c \varphi_1$, $c>0$ is the positive solution of (\ref{Pa2}). This ends the proof.
\end{proof}

\begin{remark}\label{remark1}
Another consequence of Proposition \ref{bifur} and \ref{comportamentoSigma}
is the existence of an unbounded continuum $\Sigma_0$ of positive solutions of (\ref{Pa1}), namely,
$$\Sigma_0:= \{(\lambda,u) \in \R \times \mathcal{C}_0^1(\overline{\Omega});~ (\lambda, u/\lambda + u^r) \in \widehat{\Sigma}_0\}. $$
Moreover, Proposition \ref{comportamentoSigma} still remains valid if replace $\widehat{\Sigma}_0$ by $\Sigma_0$
\end{remark}

\section{Proofs of the  main Theorems }\label{sec:final}
In this section we will prove  results of existence of positive solution of (\ref{P1}), 
under suitable assumptions on $g$.
For the reader convenience we rewrite here the theorems we are going to prove.

The analysis will be done in three cases and in the first two we will use the following classical Bolzano Theorem (see for instance \cite{ArcoyaLP}).

\begin{theorem}\label{th:Bolzano}
Let $X$ be a Banach space, $I\subset \R$ be an interval and $\Sigma_0 \subset I \times X$ be a continuum.
Assume that $h:\Sigma_{0} \to \R$ is a continuous function such that for some $(\mu_1,u_1), (\mu_2,u_2) \in \Sigma_0$ it holds
$h(\mu_1,u_1)  h(\mu_2,u_2)<0$. Then there exists $(\lambda_{*},u_{*}) \in \Sigma_0$ such that $h(\lambda_{*},u_{*})=0$.
 \end{theorem}

Roughly speaking, our results depend on the  position of $\mbox{Proj}_{\R}\Sigma_0$ with respect to
$\lambda_{1}/a$.

\subsection{The case $(\lambda_1/a,\infty) \subset \mbox{Proj}_{\R}\Sigma_0$}

Our first result of existence of positive solution of \eqref{P1} deals with the case in which $(\lambda_1/a,\infty) \subset \mbox{Proj}_{\R}\Sigma_0$. To this end, we will assume that  $g$ satisfies the next hypothesis: \medskip
\begin{enumerate}[label=(g\arabic*),ref=g\arabic*,start=1]
    \item \label{g_{1}} $g:[0,\infty) \rightarrow [0,\infty)$ is a continuous function and there exists a  constant $g_0>0$ such that
    $$g(s) >g_0 \quad \forall s \in (0,+\infty),$$
\end{enumerate}

\medskip

The first one is.
\begin{taggedtheorem}{A}
Assume that $g$ satisfies \eqref{g_{1}}. If one of the assumptions of Proposition \ref{comportamentoSigma} (1) occurs, then \eqref{P1} admits at least one  solution for each  $a>g(0)\lambda_1$.
\end{taggedtheorem}
\begin{proof}
If one of the assumptions of Proposition \ref{comportamentoSigma} (1) occurs, by Corollary \ref{Existencia} and Remark \ref{remark1}, $(\lambda_1/a,\infty) \subset \mbox{Proj}_{\R}\Sigma_0$, where $\Sigma_0$ is the unbounded continuum of positive solutions of (\ref{Pa1}).

Consider the continuous map $h: \Sigma_0 \subset (0,+\infty) \times {C}_0^1(\overline{\Omega}) \rightarrow {C}_0^1(\overline{\Omega})$ defined by
$$h(\lambda,u):= \frac{1}{\lambda} - g(|\nabla u|_2^2).$$
Then the zeros of $h$  are positive solutions of (\ref{P1}). Let us apply the Bolzano Theorem to $h$.
In $(\mu_{1},u_{1}):=(\lambda_1/a,0) \in \overline{\Sigma}_0$ we have
$$
h(\mu_{1}, u_{1})= h(\lambda_1/a,0) = \frac{a}{\lambda_1} - g(0)>0 
$$
thanks to $a>\lambda_1g(0)$.

On the other hand, 
$$
\limsup_{\substack{(\lambda,u) \in \Sigma_0 \\ \lambda \rightarrow +\infty}} h(\lambda,u) =\limsup_{\substack{(\lambda,u) \in \Sigma_0 \\ \lambda \rightarrow +\infty}} \left(\frac{1}{\lambda} - g(|\nabla u|_2^2)\right) = - \liminf_{\substack{(\lambda,u) \in \Sigma_0 \\ \lambda \rightarrow+ \infty}} g(|\nabla u|_2^2)
$$
and since $g(s) \geq g_0>0$, it follows that 
$$
\limsup_{\substack{(\lambda,u) \in \Sigma_0 \\ \lambda \rightarrow+ \infty}} h(\lambda,u) \leq - g_0 <0.
$$ 
Therefore, there exists some $( \mu_{2},u_{2}) \in \Sigma_0$ such that $h(\mu_{2},u_{2}) <0$. By the Bolzano Theorem
\ref{th:Bolzano} 
we find $(\lambda_{*},u_{*}) \in \Sigma_0$ satisfying
$$
h(\lambda_{*},u_{*})= \frac{1}{\lambda_{*}} - g(|\nabla u_{*}|_2^2)  = 0.
$$
In particular, $(\lambda_{*},u_{*})$ is a positive solution of (\ref{Pa1}) with $1/\lambda_{*} = g(|\nabla u_{*}|_2^2)$, that is,
\begin{equation*}
    \left\{ \begin{array}{ll}
         -\Delta (g(|\nabla {u_{*}}|_2^2){u_{*}} +{u_{*}}^r) = a {u_{*}} + b {u_{*}}^p&\mbox{in } \Omega,  \\
         {u_{*}} = 0& \mbox{on }\partial\Omega.
    \end{array} \right.
\end{equation*}
Thus, $u_{*}$ is a positive solution of \eqref{P1}.
\end{proof}

\subsection{The case $(0,\lambda_1/a) \subset \mbox{Proj}_{\R}\Sigma_0$.}
Now, we will prove our second result of existence of  solutions to (\ref{P1}) that deals with the case in which $(0,\lambda_1/a) \subset \mbox{Proj}_{\R}\Sigma_0$. For this, we will assume that: \medskip

\begin{enumerate}[label=(g\arabic*),ref=g\arabic*,start=2]
    \item \label{g_{2}} $g:[0,\infty) \rightarrow [0,\infty)$ is a bounded continuous function such that $g(0)>0$.
\end{enumerate}
\medskip
Then we have

\begin{taggedtheorem}{B}
Assume that $g$ satisfies \eqref{g_{2}}.
\begin{enumerate}
    \item[(i)] If $r=p$ and $b>\lambda_1$ then problem (\ref{P1}) admits (at least) one  solution for each $a< g(0)\lambda_1$.
    \item[(ii)] If $b>0$, $r<p$, $p/r <(N+2)/(N-2)$ and $g(0)\lambda_1>\phi(s_0)$ then problem (\ref{P1}) admits (at least) one  solution for each $a \in (\phi(s_0),g(0)\lambda_1)$. 

\noindent    Recall that $\phi$ and $s_{0}$ are defined in
    \eqref{eq:phi} and \eqref{eq:s0}.
\end{enumerate}
\end{taggedtheorem}
\begin{proof}
In both cases, by Corollary \ref{Existencia} and Remark \ref{remark1}, $(0,\lambda_1/a) \subset \mbox{Proj}_{\R}\Sigma_0$. Thus:\\
If (i) occurs then in $(\mu_{1},u_{1}):=(\lambda_1/a,0) \in \overline{\Sigma}_0$ we have
$$
h(\mu_{1}, u_{1})= h(\lambda_1/a,0) = \frac{a}{\lambda_1} - g(0)<0,
$$
thanks to $a<g(0)\lambda_1$. On the other hand, since $g$ is bound,
$$
\liminf_{\substack{(\lambda,u) \in \Sigma_0 \\ \lambda \rightarrow 0^+}} h(\lambda,u) =\liminf_{\substack{(\lambda,u) \in \Sigma_0 \\ \lambda \rightarrow 0^+}} \left(\frac{1}{\lambda} - g(|\nabla u|_2^2)\right) = +\infty.
$$
Therefore, there exists some $(\mu_1,u_2) \in \Sigma_0$ such that $h(\mu_2,u_2)>0$. By Bolzano Theorem \ref{th:Bolzano} we find $(\lambda_*,u_*) \in \Sigma_0$ satisfying $h(\lambda_*,u_*)=0$, thus $u_*$ is a positive solution of (\ref{P1}).\\
Similarly, if (ii) occurs  we have
$$
h(\lambda_1/a,0) = \frac{a}{\lambda_1} - g(0)<0 \quad \mbox{and} \quad \liminf_{\substack{(\lambda,u) \in \Sigma_0 \\ \lambda \rightarrow 0^+}} h(\lambda,u)=+\infty.
$$
By Bolzano Theorem \ref{th:Bolzano} we find $(\lambda_*,u_*) \in \Sigma_0$ satisfying $h(\lambda_*,u_*)=0$, thus $u_*$ is a positive solution of (\ref{P1}).
\end{proof}

\subsection{The case  $\Sigma_0 = \{\lambda_1/a\}$}
To finish, we will show  a result of existence of positive solution of (\ref{P1}) when $\Sigma_0 = \{\lambda_1/a\}$.  In this case we will not use the Bolzano Theorem \ref{th:Bolzano} and the only assumption on $g$ is that it is a continuous positive function.

\begin{taggedtheorem}{C}
Assume that $g:[0,\infty) \rightarrow [0,\infty)$ is continuous and positive. If $b=\lambda_1$ and $r=p<2$, then problem \eqref{P1} admits a positive solution if, and only if, $a\in \lambda_1 R[g] :=\{\lambda_1g(s); s \geq 0\} \subset (0,\infty)$.
\end{taggedtheorem}
\begin{proof}
Since $b=\lambda_1$ and $r=p$, it follows from Corollary \ref{Existencia} and Remark \ref{remark1} that $\Sigma_0 = \{\lambda_1/a\}$. Moreover, all positive solutions of (\ref{P1}) are given by
$$
u_c:= q_{\lambda_1/a}(c \varphi_1) \Longleftrightarrow \frac{a}{\lambda_1}u_c + u_c^r = c \varphi_1 \quad c \in \R, ~c>0.
$$
Thus,
\begin{eqnarray*}
\nabla(au_c + \lambda_1 u_c^r) = ( a + \lambda_1 ru_c^{r-1}) \nabla u_c &=& c \lambda_1 \nabla \varphi_1 \\
 \nabla u_c &=& \frac{c\lambda_1}{ a + \lambda_1 ru_c^{r-1}} \nabla \varphi_1\\
  |\nabla u_c|^2 &=&  \left(\frac{c\lambda_1}{a + \lambda_1 ru_c^{r-1}}\right)^2 |\nabla \varphi_1|^2
\end{eqnarray*}
and
\begin{eqnarray*}
|\nabla u_c|_2^2 &=&  \int_\Omega\left(\frac{c\lambda_1}{a + \lambda_1 rq_{\lambda_1/a}(c\varphi_1)^{r-1}}\right)^2 |\nabla \varphi_1|^2\\
&=& \int_\Omega\left(\frac{\lambda_1}{a/c + \lambda_1 rc^{r-2}(q_{\lambda_1/a}(c\varphi_1)/c)^{r-1}}\right)^2 |\nabla \varphi_1|^2.
\end{eqnarray*}
Consequently, since $r<2$, the map $c \in [0,\infty)\mapsto |\nabla u_c|_2 \in [0,\infty)$  is continuous and increasing. Let us prove that it is one-to-one. Indeed, by monotonicity, it follows that $c \mapsto |\nabla u_c|_2$ is an injection. To prove that it is a surjection,  it is sufficient to show that 
$$
\lim_{c \rightarrow +\infty} |\nabla u_c|_2 = +\infty.
$$
Indeed, it follows from Lemma \ref{lem:q},
$$q_{\lambda_1/a}(c\varphi_1)^{r-1} \leq \lambda_1 (c \varphi_1)^{r-1}/a \leq \lambda_1 c^{r-1}/a.$$
Thus,
\begin{eqnarray*}
|\nabla u_c|_2^2 =  \int_\Omega\left(\frac{c\lambda_1}{a + \lambda_1 rq_{\lambda_1/a}(c\varphi_1)^{r-1}}\right)^2 |\nabla \varphi_1|^2 &\geq&  \int_\Omega\left(\frac{c\lambda_1}{a +  r\lambda_1^2 c^{r-1}/a}\right)^2 |\nabla \varphi_1|^2 \\
&=&\left(\frac{\lambda_1}{ac^{-1} +  r\lambda_1^2 c^{r-2}/a}\right)^2.
\end{eqnarray*}
Once that $r<2$, one can infer
$$
\lim_{c \rightarrow+\infty} |\nabla u_c|_2 = +\infty.
$$
Subsequently, for each $a \in \{\lambda_1g(s); s \geq 0\}$, there exists $s'\geq0$ such that
$$
a=\lambda_1 g(s').
$$
By the previous discussion, we can choose $c>0$ satisfying 
$$
|\nabla u_c|_2 = s'.
$$
Therefore
$$
a= \lambda_1 g(|\nabla u_c|_2) \Longleftrightarrow \frac{1}{\lambda_1/a} = g(|\nabla u_c|_2)
$$
and $u_c$ is a  solution of (\ref{P1}). 
\end{proof}

\end{document}